\documentclass{article}
\usepackage{dcolumn}

\usepackage{verbatim}       

\usepackage[dvips]{graphicx}
\usepackage{amsthm}
\usepackage{amssymb}
\usepackage{bm}
\usepackage{amstext}
\usepackage{psfrag}
\usepackage{amsmath}
\usepackage{amsfonts}
\usepackage{mathrsfs}

\newcommand{\p}{\partial}

\newcommand{\dd}{{\rm d}}
\newcommand{\bd}{\begin{definition}}                
\newcommand{\ed}{\end{definition}}                  
\newcommand{\bc}{\begin{corollary}}                 
\newcommand{\ec}{\end{corollary}}                   
\newcommand{\bl}{\begin{lemma}}                     
\newcommand{\el}{\end{lemma}}                       
\newcommand{\bp}{\begin{proposition}}            
\newcommand{\ep}{\end{proposition}}                
\newcommand{\bere}{\begin{remark}}                  
\newcommand{\ere}{\end{remark}}                     

\newcommand{\bt}{\begin{theorem}}
\newcommand{\et}{\end{theorem}}

\newcommand{\be}{\begin{equation}}
\newcommand{\ee}{\end{equation}}

\newcommand{\bit}{\begin{itemize}}
\newcommand{\eit}{\end{itemize}}
\newtheorem{theorem}{Theorem}[section]
\newtheorem{corollary}[theorem]{Corollary}

\newtheorem{lemma}[theorem]{Lemma}
\newtheorem{proposition}[theorem]{Proposition}
\theoremstyle{definition}
\newtheorem{definition}[theorem]{Definition}
\theoremstyle{remark}
\newtheorem{remark}[theorem]{Remark}

\begin{document}
%

\title{A metrical approach to Finsler geometry}

\author{E. Minguzzi \footnote{Dipartimento di Matematica e Informatica ``U. Dini'', Universit\`a degli Studi di Firenze,  Via
S. Marta 3,  I-50139 Firenze, Italy. E-mail:
ettore.minguzzi@unifi.it}}

\date{}
\maketitle

\begin{abstract}
\noindent In the standard approach to Finsler geometry the metric is defined as a vertical Hessian and the Chern or Cartan connections appear as just two among many possible natural linear connections on the pullback tangent bundle. Here it is shown that the Hessian nature of the metric, the non-linear connection and the  Chern or Cartan connections can be derived from a few  compatibility axioms between metric and Finsler connection.
This result provides a metrical formulation of Finsler geometry which is well adapted to field theory, and which has proved useful in Einstein-Cartan-like approaches to Finsler gravity.
\end{abstract}

\section{Introduction}

The standard way of introducing Finsler geometry goes along the following lines. Let $M$ be a connected $n$-dimensional manifold,  $E=TM\backslash 0$  the slit tangent bundle,  $\pi\colon E\to M$ the projection,   $\mathcal{T}(M)$  the tensor bundle over $M$, and let $\pi^*(\mathcal{T}(M))$ be the pullback tensor bundle over $E$. We  denote with $VE\subset TE$ the vertical bundle over $E$, namely the set of all vectors $Z\in TE$ such that $\pi_*(Z)=0$.  Since $E_x:=\pi^{-1}(x)=T_xM\backslash 0$, is a vector space minus its origin, we have that for every $X\in E_x$, $T_X E_x=V_X E$ is naturally isomorphic to $T_xM$. Thusm for each $X\in E$, there is an isomorphism $v: T_xM\to V_XE$, $u \to u^v$ called the {\em vertical lift}.

(Pseudo-)Finsler geometry is concerned with a special type of metric $g\colon E\to \pi^*(T^*M\otimes_M T^*M)$, namely one for which there is  a continuous function $\mathscr{L}\colon E\to \mathbb{R}$, called {\em Finsler Lagrangian}, such that
\begin{itemize}
\item[(a)] $\mathscr{L}$ is $C^\infty$ on $E$,
\item[(b)] $\mathscr{L}(sy)=s^2 \mathscr{L}(y)$, for every $s>0$ and $y\in E$,
\item[(c)] for each $y\in E$, $u,v\in T_{\pi(y)}M$,
\[
g_y(u,v):=\frac{\p^2}{\p r\p s} \mathscr{L}(y+ru+sv) \Big|_{r=s=0}.
\]
is non-degenerate.
\end{itemize}
One speaks of `Finsler geometry' if $g$ is positive definite or of `Lorentz-Finsler geometry' if $g$ has Lorentzian signature (by continuity, as $M$ is connected, the signature is independent of the point). In the former case one might equivalently work with  the {\em Finsler fundamental function}  $F: E\to [0,\infty)$ defined by $\mathscr{L}=\tfrac{1}{2} F^2$. It can be observed that the condition ``$g$ is positive definite'' is equivalent to the condition ``$\mathscr{L}>0$ on $E$'' \cite{warner65,lovas10}. In order to simplify the terminology, in what follows  we shall generically speak of `Finsler geometry'  though the signature of $g$ will really play no role in our arguments.

We shall be particularly interested on $\pi^*(\mathcal{T}(M))$-valued forms on $E$.
There are some natural objects of this type. A natural vector-valued 0-form is the (radial)  Liouville vector field $y$ defined by $y:=X$ at $X\in E$, where we used $\pi^{-1}(x)=T_xM\backslash 0$, $x=\pi(X)$.
An example of vector-valued 1-form is provided by $\omega(Y):=\pi_*(Y)$, where $Y\in T_XE$ and $\pi_{*}\colon TE\to TM$.
As a third example, a metric $g$ on $E$ is a (0-form) section of $\pi^*(T^*M\otimes_M T^*M)$.

Let us introduce the Cartan torsion
\begin{equation} \label{cqz}
C(u,v,w)=\frac{1}{2} \frac{\p^3}{\p r\p s\p t} \mathscr{L}(y+ru+sv+tw) \Big|_{r=s=t=0}
\end{equation}
which is clearly totally symmetric, and let us consider a {\em Finsler connection}, namely a Koszul connection $\nabla$ on the pullback vector bundle $\pi^*(TM)\to E$.
 Due to the existence of the vertical lift, an equivalent approach using the vertical  bundle is possible.

Given a Finsler connection, let us consider the map $TE\to \pi^*(TM)$, $Z\to \nabla_Z y$.
The Finsler connection is said to be  {\em regular} (or {\em good})  \cite{szilasi89,anastasiei96,abate96}, if for every $X\in E$,  $x=\pi(X)$,  the  restricted map $V_XE \to T_xM$, $Z\to \nabla_Z y$, is injective (so if we lift to $V_XE$, the map  $Z\to (\nabla_Z y)^v$ is an injective endomorphism); equivalently,  defined $HE:=\textrm{ker} \nabla y$ we have $TE=HE\oplus VE$ (hence the dimension of $HE$ is $n$). It is {\em strongly regular} if the map $V_XE\to V_XE$,  $Z\to (\nabla_Z y)^v$, is the identity for every $X\in E$.

Strong regularity implies regularity. Every $n$-dimensional distribution $HE$ that splits $TE$ as above is an {\em Ehresmann (non-linear) connection}, thus each  regular Finsler connection induces a non-linear connection  $HE:=\textrm{ker} \nabla y$.

The Finsler connection can be used to extend the exterior differential $\dd$ acting on forms on $E$ into the covariant exterior differential $D$ acting on $\mathcal{T}(M)$-valued forms on $E$ by imposing the identity $D (T\otimes \beta)=\nabla T\wedge \beta+T\otimes \dd \beta$, where $T$ is a section of $\pi^*(\mathcal{T}(M))$ and $\beta$ is a form on $E$.

The curvature is a $TM\otimes_M T^*M$-valued 2-form on $E$  defined by the identity
\[
D^2 X=R(X) ,
\]
where $X$ is any section of $\pi^*(TM)\to E$. The contraction
\[
R(y)=D^2y ,
\]
is of special interest in the context of the non-linear connection $HE$ induced by $\nabla$  as one of its contributions is the {\em non-linear curvature} (see below). Notice that $\omega$ is not the only natural $TM$-valued 1-form on $E$, there is also $\bar \omega:=D y=\nabla y$. Now, since $VE=\textrm{ker}  \omega$ the regularity condition reads $ \textrm{ker} \bar \omega\cap  \textrm{ker}  \omega =0$, the zero section, or equivalently: for every $X\in TE\backslash 0$, $\bar \omega(X)\ne 0$ or $ \omega(X)\ne 0$. The latter formulation means that the components of $\omega$ and $\bar \omega$ provide a $2n$-dimensional basis for $T^*E$.

The $TM$-valued 2-forms on $E$,  $T=D\omega$ and $\bar T=D\bar\omega=R(y)$ are called {\em horizontal and vertical torsions}, respectively.
We have the first Bianchi identities
\[
D T=R\underset{\boldsymbol{\cdot}{}}{\wedge} \omega, \qquad D \bar T=R\underset{\boldsymbol{\cdot}{}}{\wedge} \bar \omega ,
\]
and second Bianchi identity
\[
D R=0.
\]

\begin{remark}
When tensor indices are omitted, contractions are emphasized using a dot, see e.g.\ the above first Bianchi identities. We stress that the dot does not represent a metric, that instead is always displayed if present (this is convenient when the connection is not metric compatible). So, for instance, $g(u,v)$ can also be written $u\cdot g \cdot v$.

When it comes to use formulas in components we might use some common abuse of notation, for instance, given an endomorphism-valued $k$-form $T=T^a{}_b e_a\otimes e^b$ its exterior derivative in components will be denoted $D T^a{}_b$ instead of $(D T)^a{}_b$.
\end{remark}

In his study of the equivalence problem in Finsler geometry Chern proved a result \cite{chern48,chern92b,bao93,spiro99} which can be phrased as follows
\begin{theorem} Suppose that $g$ is the vertical Hessian of a Finsler Lagrangian, and let $C$ be defined as in Eq.\ (\ref{cqz}).
There is a unique Finsler connection $\nabla$ with the following properties
\begin{itemize}
\item[(i)]  Regularity,
\item[(ii)] $T=0$, \qquad  (no horizontal torsion)
\item[(iii)]  $ D g(X)(u,v)=2 C(u,v,D y(X))$. \qquad  (almost metric compatibility)
\end{itemize}
\end{theorem}
Here $y$ is the Liouville vector field, $X\in TE$ and $u,v\in TM$.

\begin{remark}
Traditional formulations of this result \cite[Thm.\ 3.2]{chern00} \cite{chern92b,bao00,shen01} pass through the introduction of local coordinates and through the definition  of the connection forms associated with $\nabla$. The above is really the same result presented in a more intrinsic language.
\end{remark}

Although we made an effort to express the foundations of Finsler geometry in rather short terms, the reader will recognize that the above presentation is not entirely elegant.
In fact, the above derivation relies on a number of assumptions that do not seem to be well motivated. The metric had to be assumed of Hessian type and the Lagrangian was assumed to be positive homogeneous. The Chern connection required the preliminary definition of the Cartan torsion. Moreover, although the defining equations are few in number, they are certainly not the only equations that would lead to a meaningful connection.
Indeed, once the Cartan torsion is defined, one can define several other Finsler connections in rather natural ways. We mention the Berwald connection, the Cartan connection and the Hashiguchi connection \cite{matsumoto86,bejancu90,anastasiei96b,szilasi14,minguzzi14c}, to name a few \cite{vincze00}.

An improvement of Chern's result was obtained by Abate in \cite[Thm.\ 2.2]{abate96} where it is observed that $(iii)$ can be replaced by\footnote{His way of obtaining the torsion is less direct. He first extends $\nabla$ to a connection for $TE\to E$ and then considers its torsion, which splits into two components corresponding to our horizontal and vertical torsions. The condition he places on the torsion is equivalent to (ii) in Chern's theorem.}
\begin{itemize}
\item[$(iii')$] $ D g(X)=0$ whenever $D_X y=0$ .
\end{itemize}
The reader can find a proof of this improved version after the proof of our Theorem \ref{kkj} below.

Although this version removes the need for introducing the Cartan torsion, it still shares many of the drawbacks of the previous version, including the need for a metric of Hessian type. It is perfectly legitimate to develop Finsler geometry from a positive homogeneous Lagrangian function $\mathscr{L}$. This approach is fundamentally equivalent to that which considers the indicatrix hypersurface, i.e.\ the subset of $TM$  determined by the equation $\vert \mathscr{L}\vert =1/2$, as the main fundamental object of Finsler geometry. It is in fact an attractive point of view, the Finsler metric $g$ being related to the affine metric of the indicatrix \cite{laugwitz11} \cite[Prop.\ 4.1]{mo10} \cite[Thm.\ 6]{minguzzi15e}. It is also a quite natural point of departure for further generalizations of Finsler geometry \cite{bryant02,sabau10}. Still this approach is more variational in nature than metrical, the function $\mathscr{L}$ serving to define the length of curves and the geodesics via a stationary action.

Another elegant approach which is worth mentioning is that initiated by Klein, Voutier and Grifone \cite{klein68,grifone72,godbillon69} and pursued by Szilasi, Vincze and others \cite{szilasi98,szilasi00,szilasi14,youssef14}. Here Finsler geometry is built from a symplectic form $\Omega$ over $E$. This approach is  tailored to the geometry of the tangent bundle, so its methods depart from the typical pullback tangent bundle approach of traditional Finsler geometry.

Chern once summarized the idea behind Finsler geometry as follows  \cite{chern96}: {\em  Finsler geometry is just  Riemannian geometry without the quadratic restriction}. This claim could be really justified if we could remove the Hessian assumption on the metric. Ideally, in a metrical approach Finsler geometry should be determined from compatibility conditions between just two objects, the metric $g$ and the connection $\nabla$. No other forms of derivative, however natural, should be introduced (so the vertical derivative used above to define the Cartan torsion should be avoided).

An argument that motivates the search for a metrical approach is the following. In the usual non-metrical approach one could also have the justified concern  that manipulations of the structure equations might always leave aside further conditions implicit in the functional dependence of $g$ on $\mathscr{L}$.  On the contrary, in the sought metrical approach one could not lose track
of the Hessian nature of $g$, for that would be a consequence of the structure/compatibility equations themselves.

This type of concern is not merely academic, rather it is of physical interest. In Finslerian field theories working with the right geometry is not sufficient. One needs to know what are the independent field variables. The possibility of founding Finsler geometry on the metric allows one to considering it, and possibly  the connection, as the independent field variable, much in the same way as it is done in Einstein's gravity or in its  Palatini version. Indeed, this paper has provided the foundation necessary for this type of development, an Einstein-Cartan-like anisotropic theory resting on the theory developed in this work having been developed in the recent work \cite{minguzzi20c}.

In general, one does not want to work under (however well posed) geometrical assumptions that are spoiled under a variation of  the selected field variables. So if the selected field variable is $g$ (rather than $\mathscr{L}$), and hence if we look for a metrical field theoretical description, the Hessian nature of the metric cannot be assumed a priori as it would be spoiled under a small perturbation (small variations are considered in the Lagrangian derivation of the field equations).

With the present results we prove  that this metric program can be pursued successfully, as both the Hessian nature of the metric and the Chern and Cartan connections can be recovered from a few compatibility axioms (Thms.\ \ref{kkj} and \ref{kwj}). It must be stressed that several references privilege the Chern connection on the grounds that it can be characterized in terms of the connection forms. We shall show that, as a matter of fact,  Cartan's connection admits a similar characterization valid also when  the Hessian condition is dropped.

The reader will also recognize that our characterizations are by no means differential form reformulations of known results. For instance, the Cartan connection is usually characterized by imposing  metricity, the vanishing of the hh-component of the horizontal torsion, and  the vanishing of the  vv-component of the vertical torsion \cite{ingarden93,abate96,minguzzi14c}. In fact, the vertical torsion will not enter  our characterization at all. In general, we shall not need to split the form part of the various objects in horizontal or vertical parts, so that no expressions of the form `` the vv-component of...'' or ``the vh-component of...'' will ever be used in this work. We shall favor as much as possible algebraic characterizations, for they are the most convenient in computations.

\section{Coordinate expressions and regularity}

Let us introduce some local coordinate expressions that will be useful in proofs. Let $\{x^a \}$ be local coordinates on $M$.
We have an induced basis field $x \mapsto e_a(x)$, $e_a\colon M \to TM$, namely $e_a=\p/\p x^a$ and hence a field of dual cobasis $x \mapsto e^a(x)$, $e^a\colon M \to T^*M$, $e^a=\dd x^a$,  $e^a (e_b)=\delta^a_b$.
 Then $\{x^a, y^b\}$ are induced coordinates on $E$, $y^a\colon TM\to \mathbb{R}$, $y^a(X):=e^a(X)$.
The connection 1-forms on $E$, $\omega^a{}_{b}$, determining the connection  $\nabla$ are defined by $\nabla e_b=\omega^a{}_{b} e_a$. The connection $\nabla$ is said to be positive homogeneous of degree zero if the connection forms have this property on $E$. The cocycle of the connection 1-forms under coordinate changes shows that this definition is well posed. In this case $\bar \omega=D y$ is positive homogeneous of degree one.

Given a metric we write $g=g_{ab} e^a\otimes e^b$, where $g_{ab}$ are the metric components. We shall also write $y=y^a e_a$, $\omega=\omega^a e_a$, $\bar \omega=\bar \omega^a e_a$. Here $\omega^a=\pi^*(e^a)$, in fact for $Z\in T_XE$ we get as desired $\omega(Z)=\omega^a(Z) e_a=e^a(\pi_*(Z)) e_a=\pi_*(Z)$.

For a regular connection, since $\{\omega^a,\bar \omega^a\}$ is a basis for $T^*E$ (cf.\ the discussion in the introduction), $\omega^a{}_{b}$ is expressible as a linear combination of the form
\[
\omega^a{}_{ b}=H^a{}_{bc} \omega^c+V^a{}_{bc} \bar \omega^c,
\]
where $H^a{}_{bc}$ and $V^a{}_{bc}$ are suitable coefficients. Notice that if $\nabla$ is positive homogeneous of degree zero then $H^a{}_{bc}$ is positive homogeneous of degree zero and $V^a{}_{bc}$ is positive homogeneous of degree minus one. By definition
\begin{equation} \label{vid}
\bar \omega^a=D y^a= \dd y^a+ \omega^a{}_{b} y^b=\dd y^a+ H^a{}_{bc} y^b \omega^c+V^a{}_{bc} y^b \bar \omega^c.
\end{equation}
For a regular connection we denote with $\omega_a$ the vector on $HE:=\textrm{ker} (X\mapsto \nabla_X y)$ that projects to $e_a$, namely its horizontal lift (it is uniquely determined by regularity), and set $\bar \omega_a:=\frac{\p}{\p y^a}$,  then $\{\omega_a, \bar \omega_a\}$ is a basis of $TE$. Equation (\ref{vid}) evaluated on $\omega_a$ gives, using the horizontality of these vectors and hence $\bar \omega^a(\omega_b)=0$: $(\dd y^a+ H^a{}_{bc} y^b \omega^c)(\omega_d)=0$. That is, we obtain the inclusion $\subset$ of the following formula
\[
HE=\textrm{ker} \, e_a \left(\dd y^a+ H^a{}_{bc} y^b \omega^c\right).
\]
The other inclusion follows from the fact that $T_XE\to T_xM$, 
\[
Z\mapsto  e_a \left(\dd y^a+ H^a{}_{bc} y^b \omega^c\right)(Z)
\]
is injective as it reduces to the identity in $V_XE$ (with the usual identification between $T_xM$ and $V_XE$ provided by the vertical lift), thus the kernel in the right-hand side is at most $n$-dimensional.

We conclude that  $N^a_c:=H^a{}_{bc} y^b$ are the so called non-linear connection coefficients. The horizontal lift of $e_a$ reads $\omega_a=\frac{\p}{\p x^a}-N^b_a \frac{\p}{\p y^b}$ and is also denoted $\frac{\delta }{\delta x^a}$.

The dual to the basis $\{\omega_a, \bar \omega_a\}$ is the basis $\{\omega^a,\tilde \omega^a\}$ where
\[
\tilde \omega^a= \dd y^a+N^a_b \omega^c.
\]
 Equation (\ref{vid}) states that
\begin{equation} \label{cxd}
\tilde \omega^a=Q^a_b \bar \omega^b,
\end{equation}
 where
 \[
 Q^a_b=\delta^a_b- V^a{}_{cb} y^c.
 \]
However, we cannot conclude that $\bar \omega^a=\tilde \omega^a$, or equivalently $D y^a= \dd y^a+ N^a_c \omega^c$, for this is true iff $V^a{}_{bc} y^b=0$ namely iff the map $Z\mapsto \nabla_Z y$ is the identity once restricted to $V_XE$ (with the usual identification between $T_xM$ and $V_XE$). Since $\{\omega^a,\bar \omega^a\}$ and $\{\omega^a,\tilde \omega^a\}$ are both bases for $T^*E$, the matrix $Q^a_b$ that mediates the change of basis between $\{\bar \omega^a\}$ and $\{\tilde \omega^a\}$ must be invertible.
The connection forms can be expanded in terms of $\{\omega^a,\tilde \omega^a\}$ as follows
\[
\omega^a{}_{b}=H^a{}_{bc} \omega^c+\tilde V^a{}_{bc} \tilde \omega^c,
\]
where $\tilde V^a{}_{bc}=V^a{}_{bd} (Q^{-1})^d_c$. Moreover,
\begin{equation} \label{vif}
\bar \omega^a=D y^a= \dd y^a+ \omega^a{}_{b} y^b=\dd y^a+ N^a_{c} \omega^c+\tilde V^a{}_{bc} y^b \tilde \omega^c
\end{equation}
thus the regularity condition implies that the matrix  $\delta^a_c+\tilde V^a{}_{bc} y^b$ is invertible. Its inverse is $Q$. The previous calculations show that

\begin{proposition} \label{zzq}
Under regularity the next properties are equivalent (a) strong regularity, (b) $V^a{}_{bc} y^b=0$, (c) $\tilde V^a{}_{bc} y^b=0$, (d) $\tilde \omega^a=\bar \omega^a$.
\end{proposition}

\begin{remark}
Notice that it would be incorrect to state that the regularity condition is equivalent to the invertibility of the matrix $\delta^a_c+\tilde V^a{}_{bc} y^b$ (or of the matrix $Q^a_b$) for this matrix requires  in its very definition the assumption of regularity, as one needs to know that the 1-form connection can be expanded in the basis $\{\omega^a,\tilde \omega^a\}$ (or $\{\omega^a,\bar\omega^a\}$). However, if a non-linear connection is given in advance then we know from the outset that the covectors $\{\omega^a,\tilde \omega^a\}$ are well defined and form a basis. Then the request of regularity is understood as the compatibility condition between the given non-linear connection and the connection $\nabla$, that is the condition that the former coincides with that induced by $\nabla$ by $HE:=\textrm{ker} (X\mapsto \nabla_X y)$, cf.\ \cite{minguzzi14c}.
\end{remark}

Given the importance of the notion of regularity, the next result is of particular importance.
\begin{theorem} \label{chp}
Let $g\colon E\to \pi^*(T^*M\otimes_M T^*M)$ be a symmetric bilinear form and let $\nabla$ be a connection for the bundle $\pi^*(TM)\to E$. The 2-form on $E$ given by
\[
\Omega= g_{ab} \bar \omega^a\wedge \omega^b=
D y \underset{\boldsymbol{\cdot}}{\wedge} (g \cdot \omega)
\]
is non-degenerate iff $\nabla$ is regular and $g$ is non-degenerate.
\end{theorem}

\begin{proof}
Suppose that $\nabla$ is not regular, then $\{\omega^a,\bar \omega^a\}$ do not form a basis, hence at some point $v\in E$ they span just a subspace of $T^*_vE$, so we can find a vector $X\in T_vE$ that annihilates all these 1-forms and hence $i_X\Omega=0$, which means that $\Omega$ is degenerate. Thus the non-degeneracy of $\Omega$ implies the regularity of $\nabla$, but it also implies the non-degeneracy of $g$. Indeed, suppose that there is some $v\in E$ and some vector $u=u^a e_a\in TM$ such that $g_v(u,\cdot)=0$. By regularity   we could find $X\in VE$ such that $D_X y^a=\bar \omega^a(X)= u^a$ thus $i_X \Omega=0$, a contradiction.

Conversely, if $g$ is non-degenerate and the connection is regular any vector $X\in T_vE$ can be expanded in the basis $\{\omega_a,\bar \omega_a\}$, $X=r^a\omega_a+ t^a\bar \omega_a$, and thus $i_X\Omega= -r^a g_{ab} \bar \omega^b+g_{ad}  (Q^{-1})^d_b t^b  \omega^a$. If $i_X\Omega=0$ then by regularity $r^a g_{ab}=0$ and $g_{ad}  (Q^{-1})^d_b t^b=0$, and by the non-degeneracy of $g$, $r^a=t^a=0$, namely $X=0$, that is $\Omega$ is non-degenerate.
\end{proof}

It can be observed that
\begin{align*}
\dd(y^a g_{ab} \omega^b)&=(D y^a) \wedge g_{ab} \omega^b+y^a (D g_{ab})\wedge \omega^b+ y^a g_{ab} D\omega^a
\end{align*}

For the  Chern connection the last two terms vanish,
$\dd (y^a g_{ab} \omega^b)=\Omega$, that is $\Omega$ is a symplectic form. The 1-form $y^a g_{ab} \omega^b$ is the Hilbert form.
We have seen that for a strongly regular  connection  as Chern's, $D y^a=\dd y^a+N^a_b \omega^b$, where $N^a_b=H^a_{cb} y^c$, thus

\begin{proposition}
Let $\nabla$ be a strongly regular Finsler connection that induces the same non-linear connection as the Chern connection, then $\Omega:=(D y^a) \wedge g_{ab} \omega^b$ is a symplectic form.
\end{proposition}
Of course, the Chern, Berwald, Cartan and Hashiguchi connections are of this type. These nice symplectic properties do not distinguish between Finsler connections because they really depend only on the induced non-linear connection. We shall return on this point in the next section.

We end this section by recalling the definitions of some objects that will appear in what follows.

As already mentioned, the  {\em non-linear (Ehresmann) connection} is just a splitting  $TE=VE\oplus HE$ determined by a $n$-dimensional distribution $HE$. In local coordinates this distribution is determined by the coefficients $N^a_b$ introduced previously.  The non-linear connection is said to be  positive homogeneous of degree one if so are the coefficients $N^a_b$. The cocycle of these coefficients under change of coordinates shows that this definition is well posed. We shall use the uppercase letter $N$ to denote a  non-linear connection.

The {\em curvature of the non-linear connection} is a $TM$-valued 2-form on $E$ given by $ e_a\otimes R^a{}_{bc} \tfrac{1}{2} \omega^b \wedge \omega^c$ where
\[
R^a{}_{bc}= \frac{\delta N^a_c}{\delta x^b}- \frac{\delta N^a_c}{\delta x^b}.
\]
Given a Finsler Lagrangian $\mathscr{L}$, we say that the non-linear connection is  Barthel's (we follow the terminology of \cite{laugwitz11}, this non-linear connection is sometimes named after Berwald \cite{minguzzi14c}) if $N^a_c={\p G^a }/{\p y^c} $ where
\begin{equation} \label{vka}
2 G^a(x,y)= g^{ad} \left( \frac{\p^2 \mathscr{L}}{\p x^c \p  y^d } y^c-\frac{\p \mathscr{L}}{\p x^d} \right)
 \end{equation}
 are the spray coefficients entering the geodesic equation determined by $\mathscr{L}$, cf.\ \cite{minguzzi14c}: $\ddot x^a+2 G^a(x,\dot x)=0$ or, equivalently, as a vector field on $E$
\[
G=y^a \frac{\p}{\p x^a}-2G^a(x,y) \frac{\p}{\p y^a}=y^a \left(\frac{\p}{\p x^a}-N^a_b(x,y)\frac{\p}{\p y^a} \right)=y^a \bar \omega_a.
\]
The Barthel non-linear connection is positive homogeneous of degree one.

It is useful to recall the notion of Landsberg tensor.

Given a Finsler Lagrangian $\mathscr{L}$ the {\em Landsberg tensor} is a $T^*M\otimes_M T^*M$-valued 1-form on $E$, $L:=e^a\otimes e^b L_{abc} \omega^c$, whose components are given by \cite[Sec.\ 4.5]{minguzzi14c}
\[
L_{abc}:=-\frac{1}{2} y^r g_{rd} G^d{}_{abc}, \quad \textrm{where} \quad   G^d{}_{abc}:= \frac{\p^3 G^d}{\p y^a\p y^b\p y^c} .
\]
Notice that $L_{abc}$ is totally symmetric and annihilated by $y^a$ due to the positive homogeneity of degree two of $G^a$.

We recall what are the notable Finsler connections of a Finsler Lagrangian $\mathscr{L}$. Let
\[
\Gamma^a{}_{bc}=\tfrac{1}{2} g^{an} \left[\tfrac{\delta}{\delta x^b} \,g_{nc}+\tfrac{\delta}{\delta x^c} \, g_{nb}-\tfrac{\delta}{\delta x^n}\, g_{bc} \right],
\]
where $g_{ab}$ is the vertical Hessian of $\mathscr{L}$ and $N^a_b$ entering  the derivatives $\frac{\delta}{\delta x^b}$ are the coefficients of the Barthel's non-linear connection induced by $\mathscr{L}$. Then the Chern Finsler connection is given by $H^a{}_{bc}:=\Gamma^a{}_{bc}$, $V^a{}_{bc}=0$; the Berwald Finsler connection by $H^a{}_{bc}:= G^a{}_{bc}:= {\p^2 G^a }/{\p y^b\p y^c}$, $V^a{}_{bc}=0$; the Cartan Finsler connection by $H^a{}_{bc}:=\Gamma^a{}_{bc}$, $V^a{}_{bc}=C^a{}_{bc}$; and the Hashiguchi Finsler connection by $H^a{}_{bc}:= G^a{}_{bc}$, $V^a{}_{bc}=C^a{}_{bc}$. All the notable connections are positive homogeneous of degree zero.

The following well known formula will be useful \cite[Eq.\ (50)]{minguzzi14c}
\begin{equation} \label{cgp}
G^a{}_{bc}=\Gamma^a{}_{bc}+L^a{}_{bc}.
\end{equation}

\section{Non-linear connection and its metricity}

In this and the remaining section $g$ will be non-degenerate.

The first important step consists in showing that the Hessian nature of $g$ and the Barthel non-linear connection induced by $\mathscr{L}=\tfrac{1}{2}g_y(y,y)$ can be recovered from compatibility conditions  between the metric and the Finsler connection. These conditions only constrain the Finsler connection to stay inside a certain class of Finsler connections which includes  all the notable connections. In the next sections we shall show how to strengthen the compatibility conditions to recover the Chern's and Cartan's connections.

\begin{definition}
Given the metric $g$, two regular Finsler connections $\nabla$ and $\nabla'$ are said to be equivalent, $\nabla\sim_g \nabla'$ if their connection 1-forms are related by
\begin{equation} \label{syn}
\omega'{}^a{}_b = \omega^a{}_b+ W^a{}_{b}\, \omega^b
\end{equation}
where $W:=e_a\otimes e^b W^a{}_{b}$ is a $TM\otimes_M T^*M$-valued 1-form on $E$ such that  $W_{ab}=W_{ba}$ and $W_{ab} y^b=0$. The classes of this equivalence relation are denoted by $[\nabla]_g$.
\end{definition}

\begin{remark}
The definition is well posed since the difference of two connection is a tensor, and the conditions imposed on $W$ are independent of the coordinate system, indeed they can be expressed as $W(y)=0$, $g(X,W(Y))=g(Y,W(X))$ for every $X,Y\in TM$.

The coefficients $W_{ab}$ are obtained lowering an index, hence using $g$, thus the transformation (\ref{syn}), making use of the symmetry of these coefficients, defines a notion of class $[\nabla]_g$ that does indeed depend on $g$.

We are considering regular connections $\nabla$ that are not necessarily positive homogeneous of degree zero. If we were to restrict ourselves to such connections then $W_{ab}$ would be positive homogeneous of degree zero.

For a  related investigation on Finsler connections \cite[Prop.\ 2.18]{javaloyes20}.
\end{remark}

\begin{remark}
The above proof shows that the vector-valued 1-form $\bar \omega:=Dy=e_a (\dd y^a+\omega^a{}_b y^b)$ is independent of  the chosen element in the class and so is their induced non-linear connection $HE=\textrm{ker} \bar \omega$. Similarly, the validity of the property of strong regularity does not change with the element in the class.
\end{remark}

We shall refer to Eq.\ (\ref{syn}) as {\em the symmetry}. The same transformation in which the condition $W_{ab} y^b$ has been dropped (which implies larger equivalence classes) will be referred as the {\em amplified symmetry} consistently with the terminology in \cite{minguzzi20b} (although there the symmetry was `amplified' with respect to the much smaller projective symmetry).

\begin{definition}
The {\em canonical metric connection} associated to $\nabla$ is the connection $\tilde \nabla$ whose coefficients are
\[
\tilde \omega^a{}_b=\omega^a{}_b+\tfrac{1}{2} g^{ar} D g_{rb}.
\]
\end{definition}
It is easy to check that it is indeed metric: $\tilde D g=0$.
The following result is obvious
\begin{proposition}
The connections $\nabla$ and $\tilde \nabla$ belong to the same class $[\nabla]_g$ iff $ y \cdot D g=0$. All connections $\nabla$ in the same class share the same canonical metric connection, so one can speak of canonical metric connection of $[\nabla]_g$.
\end{proposition}

\begin{proof}
The first statement is obvious. For the second statement, let $\nabla,\nabla'\in [\nabla]_g$
\[
\tilde \omega'{}_{ab}=\omega'{}_{ab}+\tfrac{1}{2}  D' g_{ab}=\omega{}_{ab}+w_{ab}+\tfrac{1}{2}  (D g_{ab}-w_{ab}-w_{ba})=\tilde \omega{}_{ab}
\]
\end{proof}

Notice that the proof shows that the canonical metric connection is really invariant under the amplified symmetry.

\begin{proposition}
If $y \cdot Dg\underset{\boldsymbol{\cdot}}{\wedge}\omega=0$ we have that $\tilde \Omega=\Omega$ and that $\nabla$ is regular iff $\tilde \nabla$ is regular.
If $y \cdot Dg=0$ we have that $\nabla$ is strongly regular iff $\tilde \nabla$ is strongly regular.
\end{proposition}

Under the former weaker condition the non-linear connections induced by $\nabla$ and $\tilde \nabla$ do not need to coincide.

\begin{proof}
We have
\[
\tilde \Omega=
\tilde D y \underset{\boldsymbol{\cdot}}{\wedge} (g \cdot \omega)= D y \underset{\boldsymbol{\cdot}}{\wedge} (g \cdot \omega)+ \tfrac{1}{2} y \cdot Dg \wedge \omega =\Omega+ \tfrac{1}{2} y \cdot Dg \wedge \omega
\]
thus if $\tfrac{1}{2} y \cdot Dg \wedge \omega =0$ we have $\tilde\Omega=\Omega$ and so their property of being non-degenerate coincide, and so, by Thm.\ \ref{chp} the regularity properties for $\tilde \nabla$ and $\nabla$ coincide. The last statement follows from the fact that $\nabla$ and $\tilde \nabla$ belong to the same class.
\end{proof}

\begin{definition}
The {\em canonical horizontal torsion} of $\nabla$ is the $TM$-valued 2-form on $E$, $\Psi=e_a \Psi^a$ where
\begin{equation} \label{vop}
\Psi^a:=T^a+ \tfrac{1}{2} g^{ar} D g_{rb} \wedge \omega^b
\end{equation}
The {\em canonical vertical torsion} of $\nabla$ is the $TM$-valued 2-form on $E$, $\bar \Psi=e_a \bar \Psi^a$ where
\begin{align} \label{vdp}
\bar\Psi^a:&=\bar T^a+ \tfrac{1}{2} g^{ar} D g_{rb} \wedge \bar \omega^b  + D(\tfrac{1}{2} g^{ar} D g_{rb} y^b)+  \tfrac{1}{4} g^{ar} D g_{rb} \wedge  g^{bs} D g_{sc} y^c
\end{align}
(notice that the last two terms vanish if $D g_{ab} y^b=0$.)\\
The {\em canonical curvature} of $\nabla$ is the $TM\otimes T^*M$-valued 2-form on $E$ given by $\tilde R=e_a\otimes e^a \tilde R^a{}_b$ where
\begin{equation}
\tilde R_{ab}=R_{[ab]}-\tfrac{1}{4} Dg_{ac} \wedge g^{cs} D g_{sb}.
\end{equation}
\end{definition}

\begin{theorem}
Let the pair $(g,\nabla)$ be given. The canonical horizontal torsion and the canonical curvature are invariant under amplified symmetries. The canonical vertical torsion is invariant under symmetries (\ref{syn}).
They are the horizontal torsion, curvature and vertical torsion  of the canonical metrical connection.
\end{theorem}

Notice that $\bar{\tilde{\omega}}^a:=\tilde D y^a=\bar \omega^a+\tfrac{1}{2} g^{ar} D g_{rb} y^b$.

\begin{proof}
The horizontal torsion and curvature of the canonical metrical connection are constructed using $\tilde D$ and $\omega$ that are amplified symmetry invariant, so they share the same property. Similarly, the vertical torsion is constructed using $\tilde D$ and $\bar{\tilde{\omega}}$ that are  symmetry invariant so it shares the same property.
The canonical horizontal torsion is
\begin{align*}
\Psi^a&= \tilde D \omega^a=\dd \omega^a+ \tilde \omega^a{}_b \wedge \omega^b =\dd \omega^a+\omega^a{}_b \wedge \omega^b +\tfrac{1}{2} g^{ar} D g_{rb} \wedge \omega^b\\
&= D \omega^a+\tfrac{1}{2} g^{ar} D g_{rb} \wedge \omega^b=T^a+ \tfrac{1}{2} g^{ar} D g_{rb} \wedge \omega^b
\end{align*}
The canonical vertical torsion is
\begin{align*}
\bar \Psi^a&= \tilde D \bar{\tilde{\omega}}^a=\dd \bar{\tilde{\omega}}^a+ \tilde \omega^a{}_b \wedge \bar{\tilde{\omega}}^b =\dd \bar{\tilde{\omega}}^a+\omega^a{}_b \wedge \bar{\tilde{\omega}}^b +\tfrac{1}{2} g^{ar} D g_{rb} \wedge \bar{\tilde{\omega}}^b\\
&= D \bar{\tilde{\omega}}^a+\tfrac{1}{2} g^{ar} D g_{rb} \wedge \bar{\tilde{\omega}}^b\\
&=\bar T^a+ D(\tfrac{1}{2} g^{ar} D g_{rb} y^b)+ \tfrac{1}{2} g^{ar} D g_{rb} \wedge \bar \omega^b+ \tfrac{1}{4} g^{ar} D g_{rb} \wedge  g^{bs} D g_{sc} y^c
\end{align*}
The canonical curvature is
\begin{align*}
\tilde R^a{}_b&= \dd \tilde \omega^a{}_b+\tilde \omega^a{}_c \wedge \tilde \omega^c{}_b=R^a{}_b +\dd( \tfrac{1}{2} g^{ar} D g_{rb})\\
&\quad +\tfrac{1}{2} g^{ar} D g_{rc} \wedge \omega^c{}_b+\omega^a{}_c\wedge \tfrac{1}{2} g^{cr} D g_{rb}+\tfrac{1}{4} g^{ar} D g_{rc} \wedge g^{cs} D g_{sb}\\
&=R^a{}_b+D(\tfrac{1}{2} g^{ar} Dg_{rb})+\tfrac{1}{4} g^{ar} D g_{rc} \wedge g^{cs} D g_{sb} \\
&=R^a{}_b+\tfrac{1}{2} g^{ar} D^2 g_{rb}-\tfrac{1}{4} g^{ar} D g_{rc} \wedge g^{cs} D g_{sb}\\
&=\tfrac{1}{2} (R^a{}_b- R_b{}^a)-\tfrac{1}{4} g^{ar} D g_{rc} \wedge g^{cs} D g_{sb}
\end{align*}
\end{proof}

\begin{proposition}
Suppose that $y \cdot Dg=0$. We have the following identity
\begin{equation} \label{vvo}
\dd \Omega=\bar \Psi \underset{\boldsymbol{\cdot}}{\wedge} (g \cdot  \omega)- \Psi \underset{\boldsymbol{\cdot}}{\wedge} (g \cdot  \bar \omega).
\end{equation}
\end{proposition}

\begin{proof}
\begin{align*}
\dd \Omega&= \dd ( \bar \omega^a  g_{ab} \wedge \omega^b)=\bar T^a  g_{ab} \wedge \omega^b-\bar \omega^a \wedge  D g_{ab} \wedge \omega^b - \bar \omega^a  g_{ab} \wedge T^b\\
&=(\bar T^a  g_{ab}+ \tfrac{1}{2}D g_{ab} \wedge \bar \omega^a  )\wedge \omega^b- \bar \omega^a  \wedge (  \tfrac{1}{2}D g_{ab} \wedge \omega^b+ g_{ab} T^b)
\end{align*}
\end{proof}

\begin{theorem} \label{xkj}
Supppose that the pair $(g,\nabla)$ satisfies the following conditions
\begin{itemize}
\item[($\alpha$)] Regularity
\item[($\beta$)] $\Psi \underset{\boldsymbol{\cdot}}{\wedge}  (g \cdot \omega)=0$,
\item[($\gamma$)]   $\Psi \cdot g \cdot y=0$,
\end{itemize}
then $g$ is the vertical Hessian of a Finsler Lagrangian $\mathscr{L}\colon E\to \mathbb{R}$, $\mathscr{L}=\tfrac{1}{2} g_y(y,y)$ (hence positively homogeneous of degree two).
Suppose that additionally
\begin{itemize}
\item[($\delta$)] $\Psi \underset{\boldsymbol{\cdot}}{\wedge} (g \cdot \bar  \omega)=0$,
\item[($\varepsilon$)]   $y \cdot Dg=0$,
\end{itemize}
then
\begin{equation} \label{dxd}
H_{abc}=\Gamma_{abc}+\Lambda_{abc},
\end{equation}
where $\Lambda_{abc}=\Lambda_{bac}$, $y^a \Lambda_{abc}=0$, strong regularity holds, $V_{abc}= V_{bac}$, $y^a  V_{abc}=0$ and the non-linear connection induced by $\nabla$ is Barthel's.

 All conditions ($\alpha$)-($\varepsilon$) are invariant under (\ref{syn})
 so $V_{abc}$ and $\Lambda_{abc}$ are no further constrained by these conditions. The Finsler connections that satisfy ($\alpha$)-($\varepsilon$) belong to a class $[\nabla]_g$ that contains
 all the notable connections, namely Berwald, Chern, Cartan and Hashiguchi.
\end{theorem}

\begin{remark} \label{vus}
 By Eq.\ (\ref{vop}) condition ($\beta$) is really equivalent to $T \underset{\boldsymbol{\cdot}}{\wedge} (g \cdot \omega)=0$.
\end{remark}

\begin{remark}
We shall see later, cf.\ Corollary \ref{tty}, that the canonical metric connection of the class $[\nabla]_g$ selected by the theorem is the Cartan's connection of $\mathscr{L}$.
\end{remark}

\begin{remark}
For another study of condition ($\varepsilon$) see \cite[Prop.\ 3.2, Eq.\ (51)]{javaloyes20}.
\end{remark}

\begin{proof}
It is easy to check that the conditions ($\alpha$)-($\varepsilon$)
are invariant under (\ref{syn}). Thus
the conditions ($\alpha$)-($\varepsilon$) determine the connection at least up to a change (\ref{syn}). We shall see later that there is indeed no other freedom.
We shall have to prove that at least one notable connection satisfies the conditions, this way all notable connections will satisfy them as they are all related by a change of type (\ref{syn}).

By ($\alpha$) the connection is regular.
Since $\{\omega^a,\tilde \omega^a\}$ and $\{\omega^a,\bar \omega^a\}$  are bases for $T^*E$, $\omega^a{}_{b}$ is expressible as  linear combinations of the form
\[
\omega^a{}_{b}=H^a{}_{bc} \omega^c+ \tilde V^a{}_{bc} \tilde \omega^c=H^a{}_{bc} \omega^c+  V^a{}_{bc} \bar \omega^c,
\]
here $H^a{}_{bc},$ $V^a{}_{bc}$ and $\tilde V^a{}_{bc}$, are suitable coefficients. We denote $\tfrac{\delta}{\delta x^c}=\omega_c$, $C_{abc}=\frac{1}{2} \frac{\p}{\p y^c} g_{ab}$.
We have, using
the expansion of the connection in the basis $\{\omega^a,\tilde \omega^a\}$, (in the first equation we use $\tilde \omega^a(\bar \omega_b)=\delta^a_b$)
\begin{align*}
Dg_{ab}&=\left(\tfrac{\delta}{\delta x^c} g_{ab}-H_{bac}-H_{abc}\right) \omega^c +(2 C_{abc}-\tilde{V}_{bac}-\tilde{V}_{abc}) \tilde \omega^c \\
T^a&=H^a{}_{bc} \omega^c \wedge \omega^b+\tilde V^a{}_{bc} \tilde \omega^c \wedge \omega^b
\end{align*}
thus
\begin{align}
\Psi_a=&\tfrac{1}{2} D g_{ab} \wedge \omega^b+g_{ab} T^b= \tfrac{1}{2} \left(\tfrac{\delta}{\delta x^c} g_{ab}-H_{bac}-H_{abc}\right) \omega^c\wedge \omega^b \nonumber \\
&+\tfrac{1}{2}(2 C_{abc}-\tilde{V}_{bac}-\tilde{V}_{abc}) \tilde \omega^c \wedge \omega^b +H_{abc} \omega^c\wedge \omega^b+\tilde V_{abc} \tilde \omega^c\wedge \omega^b  \nonumber\\
&=\tfrac{1}{2} \left(\tfrac{\delta}{\delta x^c} g_{ab}+H_{abc}-H_{bac}\right) \omega^c\wedge \omega^b+\tfrac{1}{2}(2 C_{abc}+\tilde{V}_{abc}-\tilde{V}_{bac}) \tilde \omega^c \wedge \omega^b \label{ccf}
\end{align}
From ($\beta$) we get $\tilde{V}_{abc}=\tilde{V}_{bac}$, thus
\begin{equation} \label{tso}
\Psi_a=\tfrac{1}{2} \left(\tfrac{\delta}{\delta x^c} g_{ab}+H_{abc}-H_{bac}\right) \omega^c\wedge \omega^b+ C_{abc} \tilde \omega^c \wedge \omega^b
\end{equation}
 Imposing ($\gamma$), $\Psi_a y^a=0$,
we get that  $y^a C_{abc}=y^b C_{abc}=0$. Now the function $\mathscr{L}=\frac{1}{2} g_{ab} y^a y^b$, has first vertical derivative $\frac{\p \mathscr{L}}{\p y^p}=g_{p a} y^a+  C_{abp} y^a y^b=g_{p a} y^a$ and second derivative $\frac{\p^2 \mathscr{L}}{\p y^q\p y^p}=g_{pq}+2 C_{paq} y^a=g_{pq}$.
The metric is a vertical Hessian and furthermore, since
\[
y^a \frac{\p \mathscr{L}}{\p y^a}=2 \mathscr{L},
\]
by  Euler's homogeneous function theorem, $\mathscr{L}$ is positively homogeneous of degree two while $g$ is positively homogeneous of degree zero.
Since $C_{abc}=\frac{1}{2} \frac{\p}{\p y^c} g_{ab}=\frac{1}{2} \frac{\p^3}{\p y^a\p y^b\p y^c} \mathscr{L}$, we have that $C_{abc}$ is totally symmetric.

The condition $\Psi_a \wedge \bar \omega^a=0$ implies that the first term in the right-hand side of (\ref{tso}) vanishes, that is, setting $h_{abc}=H_{abc}-H_{bac}$
\[
h_{abc}=h_{acb}+\tfrac{\delta}{\delta x^b} g_{ca}-\tfrac{\delta}{\delta x^c} g_{ba}
\]
and hence $\Psi$ reads
\begin{equation} \label{cou}
\Psi_a= C_{abc} \tilde \omega^c \wedge \omega^b.
\end{equation}
Elaborating
\begin{align*}
h_{abc}&=h_{acb}+\left(\tfrac{\delta}{\delta x^b} g_{ca}-\tfrac{\delta}{\delta x^c} g_{ba}\right)=-h_{cab}+\left(\tfrac{\delta}{\delta x^b} g_{ca}-\tfrac{\delta}{\delta x^c} g_{ba}\right)\\
&=-h_{cba}+\left(\tfrac{\delta}{\delta x^b} g_{ca}-\tfrac{\delta}{\delta x^c} g_{ba}\right)+\left(\tfrac{\delta}{\delta x^b} g_{ac}-\tfrac{\delta}{\delta x^a} g_{bc}\right)\\
&=h_{bca}+(\tfrac{\delta}{\delta x^b} g_{ca}-\tfrac{\delta}{\delta x^c} g_{ba})+\left(\tfrac{\delta}{\delta x^b} g_{ac}-\tfrac{\delta}{\delta x^a} g_{bc}\right)\\
&=h_{bac}+\left(\tfrac{\delta}{\delta x^b} g_{ca}-\tfrac{\delta}{\delta x^c} g_{ba}\right)+\left(\tfrac{\delta}{\delta x^b} g_{ac}-\tfrac{\delta}{\delta x^a} g_{bc}\right)+\left(\tfrac{\delta}{\delta x^c} g_{ab}-\tfrac{\delta}{\delta x^a} g_{cb}\right)\\
&=-h_{abc}+2\tfrac{\delta}{\delta x^b} g_{ac}-2\tfrac{\delta}{\delta x^a} g_{cb}
\end{align*}
 thus
 \begin{equation} \label{pfq}
 H_{abc}-H_{bac}=\tfrac{\delta}{\delta x^b} g_{ac}-\tfrac{\delta}{\delta x^a} g_{cb}
 \end{equation}

Let us define $\Lambda_{abc}$ as follows
\begin{equation} \label{tva}
H_{abc}=\tfrac{1}{2} \left[\tfrac{\delta}{\delta x^b} \,g_{ac}+\tfrac{\delta}{\delta x^c} \, g_{ab}-\tfrac{\delta}{\delta x^a}\, g_{bc} \right] + \Lambda_{abc} .
\end{equation}
where the first term is symmetric in $bc$. From Eq.\ (\ref{pfq}) $\Lambda_{abc}=\Lambda_{bac}$.
Using $\tilde{V}_{abc}=\tilde{V}_{bac}$
\begin{align}
Dg_{ab}&=-2\Lambda_{abc} \omega^c +2( C_{abc}-\tilde{V}_{abc}) \tilde \omega^c , \label{scf}\\
T_a&=\Lambda_{abc} \omega^c \wedge \omega^b+\tilde V_{abc} \tilde \omega^c \wedge \omega^b . \label{mbi}
\end{align}
Thus imposing now $y^a Dg_{ab}=0$ hence (remember  $\Psi_a y^a=0$) $y_a T^a=0$ we get $\tilde V^a{}_{bc} y^b=0$, so strong regularity holds (hence $Q^a_b=\delta^a_b$, $\tilde V_{abc}=V_{abc}$)
and we also get $y^a \Lambda_{abc}=0$.

Equation (\ref{tva}) can be rewritten
\begin{equation} \label{cvx}
H^a{}_{bc}=\gamma^a_{bc}+ g^{an} [ -N^d_b  C_{dnc}-N^d_c C_{dnb}+N_n^d C_{dbc} ] + \Lambda^a{}_{bc}
\end{equation}
where
\begin{equation} \label{ffg}
\gamma^a_{bc}:=\tfrac{1}{2} g^{an} \left[\tfrac{\p}{\p x^b} \,g_{nc}+\tfrac{\p}{\p x^c} \, g_{nb}-\tfrac{\p}{\p x^n}\, g_{bc} \right],
\end{equation}
is clearly positive homogeneous of degree zero and $\gamma^a_{bc}=\gamma^a_{cb}$.

Let $2 G^a:=N^a_b y^b$. Contracting Eq.\ (\ref{cvx}) with $y^b$ we get
\begin{equation} \label{raz}
N^a_c=\gamma^a_{bc} y^b-g^{an} 2 G^d  C_{dnc}.
\end{equation}

Further contraction with $y^c$ gives
\begin{equation} \label{hfx}
2 G^a=\gamma^a_{pq} y^p y^q,
\end{equation}
so $G^a$ is positive homogeneous of degree two.
Now, observe that using (\ref{ffg})
\begin{align*}
 2\frac{\p G^a}{\p y^c}=  \gamma^a_{cq} y^q+ \gamma^a_{qc} y^q-2C^a{}_{c d} \gamma^d_{pq} y^p y^q= 2 \gamma^a_{qc} y^q-4C^a{}_{c d}  G^d ,
\end{align*}
thus comparing with (\ref{raz})
\begin{equation} \label{coa}
\frac{\p G^a }{\p y^c}=N^a_c ,
\end{equation}
namely the non-linear connection is Barthel
 (in particular $N$ is positive homogeneous of degree one).
 An easy manipulation of (\ref{hfx}) gives that $G^a$ is really the spray of the Lagrangian $\mathscr{L}$, cf.\ Eq.\ (\ref{vka}).
Notice that $\Lambda_{abc}$ is symmetric in the first two indices and annihilated by $y$ in the first two indices.  Similarly $V_{abc}$ is symmetric  in the first two indices  and annihilated by $y$ in the first two indices.

The tensors  $\Lambda_{abc}$  and $V_{abc}$ can be chosen at will provided these constraints are satisfied, so with suitable choices we can recover the Cartan, Chern-Rund, Berwald or Hashiguchi connections.
\end{proof}

The significance of Theorem \ref{xkj} is best understood asking ourselves the next questions: what is the compatibility condition between a non-linear connection $N$ and the metric $g$?

At first this question does not appear to be well posed. Since the non-linear connection is not a Finsler connection, it is not possible to take the covariant derivative of the metric and impose that it vanishes (see however \cite{bucataru07}).
Still, each non-linear connection $N$ determines in a canonical way a strongly regular Finsler connection $\nabla^N$ characterized by the next choice of coefficients
\[
 H^a{}_{bc}:=N^a{}_{bc}:=\p N^a_c/\p y^b \quad \textrm{and} \quad V^a{}_{bc}=0.
 \]
One could demand $\nabla^N g =0$ or a similar condition for horizontal vectors, but such a choice would not be adequate. Indeed, a condition of that kind would not be Finslerian. In order to stay in the realm of Finsler geometry we have rather to ask that $g$ be the vertical Hessian of some function $\mathscr{L}$ and that the non-linear connection $N$ be the Barthel's non-linear  connection for $\mathscr{L}$.

We recall that  a non-linear connection is torsionless if  $\tau^a{}_{bc}:=N^a{}_{cb}-N^a{}_{bc}=0$,  cf.\ \cite{modugno91} \cite[Sec.\ 3.3]{minguzzi14c}.

\begin{definition}
A pair $(g,N)$ is said to be compatible if any among the following equivalent conditions holds:
\begin{itemize}
\item[(a)] There is a function $\mathscr{L}:E\to \mathbb{R}$ positive homogeneous of degree two such that $g$ is the vertical Hessian of $\mathscr{L}$ and $N$ is the Barthel's non-linear connection of $\mathscr{L}$.
\item[(b)] The pair $(g,\nabla^N)$, and hence any element of $(g,[\nabla^N]_g)$, satisfies conditions ($\alpha$)-($\varepsilon$) of Thm.\ \ref{xkj}.
\item[(c)] The non-linear connection $N$ is positive homogeneous, torsionless and the following equation holds
\begin{equation} \label{ccp}
\nabla^N g_{ab}=X_{abc} \omega^c+ Y_{abc} \bar \omega^c
\end{equation}
where $X_{abc}$ and $Y_{abc}$  are totally symmetric and $y^aX_{abc}=y^a Y_{abc}=0$.
\end{itemize}
\end{definition}

Notice that in $(a)$ the function $\mathscr{L}$ is uniquely determined since by positive homogeneity it must be $\mathscr{L}=\tfrac{1}{2} g_y(y,y)$.

\begin{proof}[Proof of the equivalence.]
$(a) \Rightarrow (c)$. Under $(a)$ the Finsler connection $\nabla^N$ is Berwald's. As its torsion $T$ vanishes we have indeed, $N^a{}_{cb}-N^a{}_{bc}=0$ and by using (\ref{cgp}) it is easy to obtain the well known formula
\[
\nabla^N g_{ab}=-2 L_{abc} \omega^c+ 2 C_{abc} \bar \omega^c ,
\]
thus (\ref{ccp}) holds.

$(c) \Rightarrow (b)$. As $N$ is torsionless the torsion $T$ of $\nabla^N$ vanishes, so ($\alpha$) and ($\beta$)  are satisfied (Remark \ref{vus}). Condition ($\delta$) holds because $X_{abc}$ and $Y_{abc}$ are totally symmetric.
Conditions  ($\gamma$) and ($\varepsilon$) hold because they are annihilated by $y^a$.

$(b) \Rightarrow (a)$. If the regular connection $\nabla^N$ satisfies ($\alpha$)-($\varepsilon$) then by Thm.\ \ref{xkj} $N$  is the  Barthel's non-linear connection of $\mathscr{L}=\tfrac{1}{2} g_y(y,y)$ and $g$ is its vertical Hessian.
\end{proof}

Thus for a compatible pair $(g,N)$ it might be convenient to identify the non-linear connection with the class of Finsler connections $[\nabla]_g$ selected by   Thm.\ \ref{xkj}. Although, for a compatible pair we can recover $\mathscr{L}$ and hence any object of Finsler geometry, it is particularly interesting to look at those quantities that are easily constructed from Finsler connections in $[\nabla]_g$ , but that do not depend on the representative, in other words, that share the symmetry (\ref{syn}).

The following classical related result is worth mentioning  \cite[Sec.\ II.3]{laugwitz11} \cite{szabo08}  \cite[Thm.\ II.33]{grifone72} \cite{okada82} \cite[Cor.\ 14]{minguzzi14c}.

\begin{theorem}
Let $\mathscr{L}$ be a Finsler Lagrangian (hence positively homogeneous of degree two and with non-degenerate vertical Hessian).
There is one and only one torsionless non-linear connection positive
homogeneous of degree one for which the Finsler Lagrangian $\mathscr{L}$ vanishes on the
horizontal vectors, namely the Barthel non-linear connection.
\end{theorem}

The proof of the following result is straightforward.
\begin{proposition}
Let $\nabla$ be regular Finsler connection. Let $\Lambda_{abc}$ and $\Pi_{abc}$ be defined by
\begin{equation} \label{cfp}
D g_{ab}=-2 \Lambda_{abc} \omega^c-2 \Pi_{abc} \bar \omega^c
\end{equation}
The condition
\begin{equation} \label{cck}
  \bar \omega \underset{\boldsymbol{\cdot}}{\wedge} Dg \underset{\boldsymbol{\cdot}}{\wedge} \omega=0
\end{equation}
is equivalent to the total symmetry of $\Lambda_{abc}$ and $\Pi_{abc}$.

In particular, if $\nabla$ is an element of the class $[\nabla]_g$ selected by Thm.\ \ref{xkj} then $-2 \Pi_{abc}:=2C_{abc}- 2V_{abc}$ thus $V_{abc}$ is totally symmetric.
\end{proposition}

\begin{remark}
It is possible to define another equivalence relation with smaller equivalence classes by saying that two  connection 1-forms are equivalent if related by
\begin{equation} \label{sym}
\omega'{}^a{}_b = \omega^a{}_b+ A^a{}_{bc}\, \omega^c+ B^a{}_{bc}\, \bar \omega^c
\end{equation}
where $A_{abc}$ and $B_{abc}$, are totally symmetric and annihilated by $y^a$. The notable Finsler connections of a Lagrangian $\mathscr{L}$ would still belong to the same equivalence class.
\end{remark}

The next result and proof are a continuation of those of Theorem \ref{xkj}.

\begin{proposition}
The Finsler connections that satisfy  ($\alpha$)-($\varepsilon$) also satisfy the next identities
\begin{align}
g\cdot \Psi=\tfrac{1}{2} Dg  \underset{\boldsymbol{\cdot}}{\wedge} \omega + g \cdot T&= e^a \otimes C_{abc}   \bar \omega^c \wedge \omega^b, \label{car} \\
g \cdot \bar \Psi=\tfrac{1}{2} Dg  \underset{\boldsymbol{\cdot}}{\wedge} \bar \omega + g \cdot \bar T&= \! e^a \! \otimes\! \left(R_{abc} \tfrac{1}{2}\omega^b \!\wedge \omega^c\!\!-\! L_{abc} \omega^b \! \wedge \bar \omega^c \right)  \label{caz},\\
\bar \Psi \!\underset{\boldsymbol{\cdot}}{\wedge}\! (g \! \cdot \!  \omega)=\!- \tfrac{1}{2} \bar \omega \underset{\boldsymbol{\cdot}}{\wedge} \!  Dg  \underset{\boldsymbol{\cdot}}{\wedge}  \omega + \! \bar T \! \underset{\boldsymbol{\cdot}}{\wedge} \! (g \! \cdot \!  \omega) &=0, \label{riv}\\
\bar \Psi \underset{\boldsymbol{\cdot}}{\wedge} (g \cdot \bar \omega)=\bar T \underset{\boldsymbol{\cdot}}{\wedge} (g \cdot \bar \omega)&=R_{abc} \tfrac{1}{2} \bar \omega^a\wedge \omega^b \wedge \omega^c, \label{ppo} \\
y\cdot g\cdot T&=0,  \label{pri}\\
y \cdot g \cdot \bar \Psi=y \cdot g\cdot \bar T&=0, \label{sec}\\
 y \cdot D^2 g  \underset{\boldsymbol{\cdot}}{\wedge}  \bar \omega&=0, \label{ffl} \\
 y\cdot g \cdot D \bar T&=-R_{abc} \tfrac{1}{2} \bar \omega^a\wedge \omega^b \wedge \omega^c, \label{rok}
\end{align}
where the left-hand sides of (\ref{car})-(\ref{sec}) are   invariant under  (\ref{syn}).
Other quantities invariant under (\ref{syn}) are $\omega$, $\bar \omega$, $y \cdot g\cdot \omega$ and $y \cdot g\cdot \bar \omega$ (the latter is $\dd \mathscr{L}$).
Moreover, $\Omega:=\bar \omega \underset{\boldsymbol{\cdot}}{\wedge}  (g \cdot \omega) $ is also invariant and exact as $\Omega=\dd (y \cdot g\cdot \omega)$.

Finally,  we have the identity
\[
\bar \omega \underset{\boldsymbol{\cdot}}{\wedge} Dg \underset{\boldsymbol{\cdot}}{\wedge}  \omega = 2\bar T \underset{\boldsymbol{\cdot}}{\wedge}  (g \cdot \omega)=-y \cdot D^2 g  \underset{\boldsymbol{\cdot}}{\wedge} \omega=2y\cdot g \cdot DT .
\]
\end{proposition}

\begin{remark}
Due to Eq.\ (\ref{caz}), the equations (\ref{riv}) and (\ref{sec}) provide results for the induced non-linear curvature, namely $R_{[abc]}=0$ and $y^a R_{abc}=0$.
\end{remark}

\begin{proof}
From ($\varepsilon$) we have
\[
\dd \mathscr{L}=\dd \left(\tfrac{1}{2} y \cdot g \cdot y\right)=y \cdot g \cdot \bar \omega+\tfrac{1}{2} y \cdot D g \cdot y=y \cdot g \cdot \bar \omega
\]
Equation (\ref{pri}) is immediate from ($\gamma$) and ($\varepsilon$).
Equation (\ref{sec}) follows from ($\varepsilon$), indeed
\[
y^a g_{ab} \bar T^b=\dd (y^ag_{ab} \bar \omega^b)-\bar \omega^a\wedge g_{ab} \bar \omega^b-y^a D g_{ab} \wedge \bar \omega^b= \dd^2 \mathscr{L}-0-0=0
\]

From ($\varepsilon$) and $y\cdot g\cdot T =0$  we get
\[
\dd \left(y^a g_{ab} \omega^b\right)=\bar \omega^a \wedge g_{ab} \omega^b+y^a D g_{ab} \wedge \omega^b+ y^a g_{ab}  T^b=\bar \omega^a \wedge g_{ab} \omega^b=\Omega.
\]
Thus $\dd \Omega=0$, so  from Eq.\ (\ref{vvo}) and ($\delta$) we get (\ref{riv}).
Using ($\varepsilon$)
\begin{align*}
0=\dd \left(y^a Dg_{ab} \wedge \omega^b\right)&= \bar \omega^a \wedge  Dg_{ab} \wedge  \omega^b+ y^a D^2g_{ab} \wedge \omega^b-y^a Dg_{ab}
\wedge T^b\\
&= \bar \omega^a \wedge  Dg_{ab} \wedge  \omega^b+ y^a D^2g_{ab} \wedge \omega^b .
\end{align*}
Similarly, using   ($\varepsilon$)
\begin{align*}
0=\dd \left(y^a Dg_{ab} \wedge \bar \omega^b\right)&= \bar \omega^a \wedge  Dg_{ab} \wedge  \bar \omega^b+ y^a D^2g_{ab} \wedge \bar \omega^b-y^a Dg_{ab}
\wedge \bar T^b\\
&= - y^a D^2g_{ab} \wedge \bar \omega^b,
\end{align*}
which proves  (\ref{ffl}).
Using (\ref{pri})  we get
\begin{align*}
0&=\dd \left(y^a  g_{ab} T^b\right)=\bar \omega^a \wedge g_{ab} T^b+y^a D g_{ab} \wedge T^b+ y^a  g_{ab}  D T^b\\
&= g_{ab} \Psi^b \wedge \bar \omega^a -\tfrac{1}{2} \bar \omega^a \wedge  Dg_{ab} \wedge   \omega^b + y^a  g_{ab}  D T^b \\
&=-\tfrac{1}{2} \bar \omega^a \wedge  Dg_{ab} \wedge   \omega^b + y^a  g_{ab}  D T^b .
\end{align*}

By the strong regularity of the connection, Equation (\ref{car}) was obtained in Eq.\ (\ref{cou}). Let us denote $G^a{}_{bc}=\frac{\p}{\p y^b}N^a_c$.
Since $\bar \Psi^a$ does not depend on the representative of the class $[\nabla]_g$ we can calculate it using the Cartan connection for which $V_{abc}=C_{abc}$ is totally symmetric. Using Eq. (\ref{vdp}) and ($\varepsilon$)
\begin{align*}
\bar \Psi^a=\bar T^a&=D(Dy^a)=\dd (Dy^a)+\omega^a{}_b \wedge Dy^b=\dd (\dd y^a +N^a_c \omega^c)+\Gamma^a{}_{bc} \omega^c \wedge \bar \omega^b\\
&= \left(\frac{\delta N^a_c}{\delta x^b} \omega^b+G^a{}_{bc} \bar \omega^b\right) \wedge \omega^c+\Gamma^a{}_{bc} \omega^c \wedge \bar \omega^b\\
&= \frac{\delta N^a_c}{\delta x^b} \omega^b \wedge \omega^c+ ( G^a{}_{bc}-\Gamma^a{}_{bc})  \bar \omega^b \wedge \omega^c=R^a{}_{bc} \tfrac{1}{2} \omega^b \wedge \omega^c+ L^a{}_{bc} \bar \omega^b \wedge \omega^c
\end{align*}
so we get Eq.\ (\ref{caz}).
Equation (\ref{ppo}) is immediate from (\ref{caz}).
Equation (\ref{rok}) follows from (\ref{ppo}) and (\ref{sec})
\[
0=\dd \left(y^a g_{ab} \bar T^b\right)=\bar \omega^a \wedge  g_{ab} \bar T^b+y^a D g_{ab} \wedge T^b+y^a g_{ab}  D \bar T^b= \bar \omega^a \wedge  g_{ab} \bar T^b+y^a g_{ab} D \bar T^b .
\]
\end{proof}

It is quite interesting that expressions (\ref{car}) and (\ref{caz}) do not depend on the representative and that they involve some of the most interesting quantities in Finsler geometry, namely the Cartan torsion, the non-linear curvature and the Landsberg tensor. This fact proves that these quantities live, so to say,  in the geometry of compatible pairs $(g,N)$ rather than in $(g,\nabla)$. Another invariant is the canonical curvature. An invariant related to it appeared in a joint work with A.\ Garc\'{\i}a-Parrado \cite{minguzzi20b} where we introduced the amplified symmetry.  In that work we presented an application in the context of Einstein-Cartan gravity theory  but a further work with application to Finsler gravity will soon appear \cite{minguzzi20c}.
The importance of this curvature invariant stands in the fact that   a  Finsler gravity theory can be obtained by implementing it in a Finslerian action that  suitably generalizes Einstein-Hilbert's \cite{minguzzi20c}. Our previous discussion can then be used to argue that the field variables can be identified with $(g,N)$ instead of $(g,\nabla)$.

In the next sections we characterize the Chern and Cartan connections. It is worth pointing out that these Finsler connections admit nice characterizations in terms of foliation theory \cite{feng13}.

\subsection{The Chern connection}

The next result removes the Hessian metric assumption from Chern's theorem. Furthermore, the only type of derivation that appears in the statement is the covariant exterior differential. There is no need to preliminarily define the Cartan torsion by means of a vertical derivative. This version seems  to be the best suited for computations.
\begin{theorem} \label{kkj}
Let $g$ be a metric on the pullback tangent bundle.
There is only one  Finsler connection  with the following properties
\begin{itemize}
\item[($\tilde \alpha$)] Regularity,
\item[($\tilde \beta$)] $T=0$, \qquad \qquad  \quad (no horizontal torsion)
\item[($\tilde \gamma$)]  $y \cdot Dg=0$, \qquad  \quad (positive homogeneity)
    \item[($\tilde \delta$)] $Dg \underset{\boldsymbol{\cdot}{}}{\wedge} Dy=0$. \quad  \ \ (almost metric compatibility)
\end{itemize}
Moreover, under these conditions $g$ is the vertical Hessian of a function $\mathscr{L}\colon E\to \mathbb{R}$ positive homogeneous of degree 2, and the  connection is Chern's. Finally, these results do not change if ($\tilde \delta$) is replaced by
\begin{itemize}
\item[($\tilde \delta'$)]   $y\cdot D^2 g=0$.
\end{itemize}
\end{theorem}

The almost metric compatibility condition appearing in this statement does not coincide with the condition   $(iii)$ appearing in Chern's original theorem, or with the condition $(iii')$ in Abate's version commented previously.
They are related as follows $(iii)\Rightarrow (\tilde\delta) \Rightarrow (iii')$. The first implication is immediate using the symmetry of the Cartan torsion (if one writes $(iii)$ one should assume that the metric is of Hessian type),   for the second implication it is sufficient to take the interior product with an arbitrary horizontal vector. Still, as explained below, as a whole the conditions of this theorem are no stronger than Abate's.

\begin{proof}
The equivalence of ($\tilde \delta$) and ($\tilde\delta'$) is simple, in fact taking the covariant exterior differential of ($\tilde\gamma$)
\[
y\cdot D^2 g+Dg \underset{\boldsymbol{\cdot}}{\wedge} Dy=0.
\]
With reference to Thm.\ \ref{xkj}, it is clear that $(\tilde \alpha) = (\alpha)$, $(\tilde \gamma)=(\varepsilon)$, $(\tilde \beta)$ and  $(\tilde \gamma) \Rightarrow (\gamma)$,  $(\tilde \beta) \Rightarrow (\beta)$, $(\tilde \beta)$ and  $(\tilde \delta) \Rightarrow (\delta)$.
The assumptions of Thm.\ \ref{xkj} are satisfied, thus $g$ is the vertical Hessian of a function $\mathscr{L}\colon E\to \mathbb{R}$ positive homogeneous of degree 2 and the non-linear connection induced by $\nabla$ is Barthel's. From Eq.\ (\ref{cfp}) and ($\tilde \delta$) we get $\Lambda_{abc}=0$ and from Eq.\ (\ref{dxd}) we obtain $H^a{}_{bc}=\Gamma^a{}_{bc}$. From Eq.\ (\ref{mbi}) and $(\tilde \beta)$ we get $V^a{}_{bc}=0$. Hence the connection is Chern's.
\end{proof}

\begin{proof}[Proof of Abate's version via Theorem \ref{kkj}]
Regularity and the condition $T=0$ imply that $H^a{}_{bc}=H^a{}_{cb}$ and $V^a{}_{bc}=0$, thus the connection is strongly regular.  We notice that
\[
D g_{ab}=\left(\tfrac{\delta}{\delta x^c} g_{ab} -H^d{}_{ac} g_{db}-H^d{}_{bc} g_{ad}\right) \omega^c+2C_{abc} \bar \omega^c
\]
thus if one assumes that $g$ is a vertical Hessian of a function positive homogeneous of degree two, $C_{abc}$ is totally symmetric and annihilated by $y$ and hence $(\tilde \gamma)$ and $(\tilde \delta)$ follow from the assumption that $D_X g=0$ for $X$ horizontal. Thus the assumptions of Theorem \ref{kkj} hold.
\end{proof}

\subsection{The Cartan connection}
We are ready to present our characterization of the Cartan connection
\begin{theorem} \label{kwj}
There is only one  Finsler connection  with the following properties
\begin{itemize}
\item[($\hat \alpha$)] `regularity and $y \cdot g \cdot T=0$' or strong regularity,
\item[($\hat \beta$)] $T \underset{\boldsymbol{\cdot}{}}{\wedge} (g\cdot \omega)=0$, \qquad \qquad   (almost no horizontal torsion)
\item[($\hat \gamma$)]  $T \underset{\boldsymbol{\cdot}{}}{\wedge} (g \cdot  D y)=0$, \qquad  \quad (horizontal torsion symmetry)
\item[($\hat \delta$)]   $Dg=0$. \qquad \qquad \qquad \ \  (metricity)
\end{itemize}
Moreover, under these conditions $g$ is the vertical Hessian of a function $\mathscr{L}\colon E\to \mathbb{R}$ positive homogeneous of degree 2, and the  connection is Cartan's. Finally, under the option  `regularity and $y \cdot g \cdot T=0$' in $(\hat \alpha)$ we can replace ($\hat \gamma$) with
\begin{itemize}
\item[($\hat \gamma'$)] $y \cdot g \cdot  R \underset{\boldsymbol{\cdot}{}}{\wedge} \omega$=0.
\end{itemize}
\end{theorem}

\begin{proof}
By ($\hat\alpha$) the connection is regular. The last statement follows from
\[
\dd \left(y \cdot g \cdot  T\right)=\dd \left(y^a g_{ab} T^b\right)=(D y^a) \wedge g_{ab} T^b+ y^a g_{ab} D T^b=  T  \underset{\boldsymbol{\cdot}{}}{\wedge} (g \cdot \bar \omega)+ y \cdot g \cdot R \underset{\boldsymbol{\cdot}{}}{\wedge}  \omega,
\]
where we used the first Bianchi identity. Since $T^a=H^a{}_{bc} \omega^c \wedge \omega^b+V^a{}_{bc} \bar \omega^c \wedge \omega^b$, condition ($\hat \beta$) implies $V_{abc}=V_{bac}$ thus $y \cdot g \cdot T=0$ implies strong regularity. Conversely, by  ($\hat \gamma$) $T^a=V^a{}_{bc} \bar \omega^c \wedge \omega^b$ thus strong regularity implies $y \cdot g \cdot T=0$.
With reference to Thm.\ \ref{xkj} we have $(\hat \alpha) \Rightarrow (\alpha)$, $(\hat \beta)$ and $(\hat \delta)$ $\Rightarrow (\beta)$, $(\hat \gamma)$  and $(\hat \delta) \Rightarrow (\delta)$, $(\hat \delta) \Rightarrow (\varepsilon)$, $(\hat \alpha)$ and $(\hat \delta) \Rightarrow (\gamma)$.

The assumptions of Thm.\ \ref{xkj} are satisfied thus $g$ is the vertical Hessian of a function $\mathscr{L}\colon E\to \mathbb{R}$ positive homogeneous of degree 2, and the  non-linear connection induced by $\nabla$ is Barthel's.  From Eq.\ (\ref{cfp}) and $(\hat \delta)$ we get $\Lambda_{abc}=  \Pi_{abc} =0$, thus from Eq.\ (\ref{dxd}) we get $H^a{}_{bc}=\Gamma^a{}_{bc}$  and from Eq.\ (\ref{scf}) we get $V^a{}_{bc}=C^a{}_{bc}$, thus the connection is Cartan's.
\end{proof}

\begin{corollary} \label{tty}
The canonical metric connection of the class $[\nabla]_g$ selected by Theorem  \ref{xkj} is the Cartan's connection of $\mathscr{L}=\tfrac{1}{2} g_y(y,y)$.
 \end{corollary}

\begin{proof}
The connection $\nabla$ satisfies the conditions $(\alpha)$-$(\varepsilon)$ and the canonical metric connection $\tilde \nabla$ satisfies these same conditions because it is related to $\nabla$ by a symmetry (\ref{syn}). The condition $(\alpha)$-$(\varepsilon)$ for $Dg=0$ imply those of Thm.\ \ref{kwj}, thus $\tilde \nabla$ is the Cartan connection of $\mathscr{L}$.
\end{proof}

It does not seem possible to obtain similar theorems for the Berwald or Hashiguchi connection, the best one can do  is expressed precisely by  Thm.\ \ref{xkj}. Existing characterizations do not respect the general philosophy of this work previously outlined \cite{abate96,okada82,minguzzi14c}.

\section{Conclusions}

In this work we recovered Finsler geometry via conditions on the pair $(g,\nabla)$ expressed by means of the covariant exterior differential and without (i) ever resorting to the splitting of tensor-valued 2-forms (or any $k$-form for the matter) on $E$ in $hh$, $vh$ or $vv$ components, and  (ii) ever using other derivatives but the covariant exterior differential (so never mentioning nor introducing horizontal or vertical covariant derivatives).

This approach appears to be  more economic than previous approaches and better suited for computations.

As a result we gave new characterizations of the Chern and Cartan connections. In particular, we clarified that the Chern connection is not the only Finsler connection that can be characterized by means of its connection forms.

We also found convenient to introduce and explore a notion of compatibility condition between the non-linear connection $N$ and the metric $g$. Our approach has emphasized the fact that the non-linear connection can be regarded as an equivalence class of Finsler connections, and that the most relevant properties of the pair $(g,N)$ are those which are independent of the representative $\nabla$, namely those invariant under the symmetry (\ref{syn}). As already mentioned, the application  to a Finsler gravity theory are presented in a different work \cite{minguzzi20c}.

\section*{Acknowledgments}   This work has been
partially supported by GNFM of INDAM.


\end{document}